\documentclass[12pt]{amsart}
\usepackage{color}
\usepackage[all]{xy}
\usepackage{amssymb,amsmath}
\usepackage{amscd,amsmath,latexsym,amssymb}
\usepackage[mathcal]{euscript}
\usepackage{tikz-cd}

\newcommand{\ba}{\overline{\alpha}}

\date{July 19, 2023 (v2)}

\calclayout

\usepackage{graphicx}

\usepackage{microtype}

\newcommand{\cL}{\mathcal L}
\newcommand{\cM}{\mathcal M}

\newcommand{\bZ}{\mathbb Z}
\newcommand{\bQ}{\mathbb Q}
\newcommand{\bC}{\mathbb C}
\newcommand{\bR}{\mathbb R}

\newcommand{\CP}{\bC P}

\newcommand{\vv}{\, | \,}

\newcommand{\iso}{\approx}

\DeclareMathOperator{\Sign}{Sign}

\DeclareMathOperator{\ad}{ad}

\DeclareMathOperator{\Hom}{Hom}

\DeclareMathOperator{\Ind}{Ind}

\DeclareMathOperator{\Image}{Im}

\DeclareMathOperator{\ch}{ch}
\DeclareMathOperator{\csch}{csch}
\DeclareMathOperator{\tr}{tr}

\newtheorem{dummy}{anything}[section]
\newtheorem{theorem}[dummy]{Theorem}
\newtheorem*{thma}{Theorem A}

\newtheorem{lemma}[dummy]{Lemma}

\theoremstyle{definition}
\newtheorem{definition}[dummy]{Definition}
  \newtheorem{example}[dummy]{Example}
  
  \newtheorem{remark}[dummy]{Remark}

\newcommand
{\eqncount}{\setcounter{equation}{\value{dummy}}%
\addtocounter{dummy}{1}}
\numberwithin{equation}{section}

\newcommand{\la}{\langle}
\newcommand{\ra}{\rangle}
\newcommand{\bd}{\partial}

\newcommand{\cy}[1]{\bZ/{#1}}
\newcommand{\nr}[1]{\medskip\noindent{{\bf #1}.}}
\newcommand{\parni}{\medskip\noindent}

\begin{document}

\title{Finite group actions on $4$-manifolds and equivariant bundles}
\author{Nima Anvari}
\address{Department of Mathematics \& Statistics, McMaster University
L8S 4K1, Hamilton, Ontario, Canada}
\email{anvarin@math.mcmaster.ca}

\author{Ian Hambleton}
\address{Department of Mathematics \& Statistics, McMaster University
L8S 4K1, Hamilton, Ontario, Canada}
\email{hambleton@mcmaster.ca}

\thanks{This research was partially supported by NSERC Discovery Grant A4000.}

\begin{abstract}
Given a $4$-manifold with a homologically trivial and locally-linear cyclic group action, we obtain necessary and sufficient conditions for the existence of equivariant bundles.  The conditions are derived from the twisted signature formula and are in the form of congruence relations between the fixed point data and the isotropy representations. 
\end{abstract}

\maketitle

\section{Introduction}\label{sec:intro}
Finite group actions on $4$-manifolds can be studied in various settings. We are  mainly interested in comparing smooth actions with those which are topological and locally linear, but important examples arise for symplectic $4$-manifolds and complex surfaces. Here is a sampling of survey articles and recent work on aspects of this general theme: \cite{Anvari:2016a,Anvari:2016,Anvari:2021,Baraglia:2023,Braam:1993,Chen:2010,Edmonds:2018,Furuta:1989,Hambleton:1989,Hambleton:1992,Hambleton:1995,Hambleton:2022,Kwasik:1993,Ruberman:1995,Wilczynski:1987}. We will focus on the existence and classification of equivariant bundles, and their applications in Yang-Mills gauge theory to the study of finite group actions.

\medskip
We begin by recalling some standard definitions.
Let $(X,\pi)$ denotean oriented, simply-connected, closed $4$-manifold $X$  together with a locally linear and homologically trivial action of a cyclic group $\pi=\bZ/p$ of prime order. The fixed point set $X^{\pi}$  has Euler characteristic $\chi(X^{\pi}) = b_2(X) + 2 $, by the Lefschetz fixed point formula. In general $X^{\pi}$ will consist of isolated fixed points and a disjoint union of fixed $2$-spheres (see \cite[\S 2]{Edmonds:1989}).

\begin{definition}
At each isolated fixed point $x \in X^{\pi}$, the tangent space admits an equivariant decomposition $(T_{x_i}X,\pi)=\bC(a_i)\oplus\bC(b_i)$  of complex representation spaces. Let  
$$t\cdot (z_1,z_2)=(\zeta^{a_i}z_1,\zeta^{b_i}z_2)$$
denote the action for $t$ a fixed generator in the cyclic group $\pi$, $\zeta=e^{2\pi i/p}$, and with  integers $(a_i,b_i)$, both non-zero modulo $p$. 
\begin{enumerate}
\item The integers $(a_i, b_i)$ are the local tangential \emph{rotation data},  and are well-defined up to order and simultaneous change in sign. 

\item Similarly, for each point $x$ on a $\pi$-fixed $2$-sphere $F_j$,  there is a representation $\bC\oplus \bC(c_j)$ corresponding to the equivariant splitting $$TX\vv_{F_j}=T F_j \oplus N(F_j)$$
 where $N(F_j)$ is the normal bundle with rotation $\zeta^{c_j}$, and $c_j \not\equiv 0 \pmod{p}$.
 
 \item The total \emph{fixed point rotation data} is the collection 
 $$\mathcal{F}=\{(a_i,b_i),(c_j,\alpha_j)\vv i \in I, j \in J \}.$$  where $I$ and $J$ index the isolated fixed points and $2$-spheres respectively and $\alpha_j=[F_j]\cdot[F_j]$ is the  self-intersection number of the fixed spheres $[F_j] \in H_2(X;\bZ)$.

 \item By an equivariant line bundle $(L,\pi)\rightarrow (X,\pi)$ we mean a principal $U(1)$-bundle $L$ over $X$ together with a lift of the $\pi$-action to the total space. Given such a lift, there exists a set of \emph{isotropy representations} of the $\pi$-action on each fiber over the fixed point set which we denote by  $L\vv_{x_i}=t^{\lambda_i}$ over isolated fixed points and $L\vv_{F_j}=t^{\lambda_j}$ over a fixed $2$-sphere. Denote the collection of these isotropy representations by $\mathcal{I}=\{\lambda_i,\lambda_j \vv i \in I, j \in J\}$. 
 \end{enumerate}
\end{definition}

With this notation we have our main result.

\begin{thma}
Let $(X,\pi)$ denote an oriented, simply-connected, closed $4$-manifold with a locally linear, homologically trivial action of a cyclic group $\pi=\bZ/p$ of odd prime order $p$, with fixed-point rotation data $\mathcal{F}=\{(a_i,b_i),(c_j,\alpha_j)\vv i \in I, j \in J \}$. 

A collection 
$\mathcal{I}$ of integers $\{\lambda_i,\lambda_j \vv i \in I, j \in J\}$ can be realized  (modulo $p$) as the isotropy representations of an equivariant line bundle $(L, \pi) \rightarrow (X, \pi)$ if and only if there is a collection of integers $\{m_j \vv j \in J\}$ such that
\begin{equation}\label{eq:thma}
\sum_{i\in I} \dfrac{\lambda_i}{a_ib_i}+\sum_{j \in J} \dfrac{c_j m_j-\lambda_j \alpha_j}{c^2_j} \equiv 0 \pmod p.
\end{equation}
When such a line bundle exists, the integers $m_j  = c_1(i^{\ast}L)[F_j]$ satisfy equation \textup{\eqref{eq:thma}}.
\end{thma}

\begin{remark}
A special case of this result can be found in  \cite[Proposition 1.4]{Hambleton:1989}, for a cyclic group $\pi$ of odd order acting locally linearly and semi-freely on the complex projective plane $\bC P^2$ which has three isolated fixed points. The necessary condition \eqref{eq:thma} in Theorem A is established by extending the methods of
\cite[\S 2]{Hambleton:1989} to more general actions. 
 \end{remark}

\begin{remark} The definitions above generalize directly to actions $(X, \pi)$ where $\pi$ is any finite group, and the structural group $G$ of the principal bundle is any compact Lie group.  The details will be left to the reader (see also \cite{Hambleton:2003,Hambleton:2010}).
\end{remark}

For applications in gauge theory, $G = SU(2)$ is an important example. In Theorem \ref{thm:sevenone} we obtain the analogous necessary conditions for the existence of equivariant $SU(2)$-bundles, and discuss the standard examples on $X = S^4$ and $X = \overline{\bC P}^2$ . In Section \ref{sec:seven} we apply our methods to compute the formal dimensions of fixed strata equivariant $SU(2)$-moduli spaces (see \cite{Hambleton:1992}), and work out one special case.

\section{Some motivating questions}\label{sec:one}
Here are some questions related to the general theme (all actions will be assumed to preserve orientation). More information about some of these directions can be found in the references.

\nr{1} 
Does there exist a smooth $\cy p$-action on a (homotopy) K3 surface, which induces the identity on integral homology ?

\medskip
This is a well-known question of Allan Edmonds. Note that results of Edmonds and Ewing \cite{Edmonds:1992}  imply that topological locally linear examples exist for odd primes.

\nr{2}  Does there exist a smooth $\cy p$-action on a homotopy K3 surface, which contains an invariant embedded Brieskorn homology $3$-sphere $\Sigma(2,3,7)$ ?

\medskip
Fintushel and Stern \cite{Fintushel:1991} showed that many homotopy K3 surfaces admit embeddings of  $\Sigma(2,3,7)$. The question concerns the possible existence of an equivariant splitting of the K3 surface along the Brieskorn sphere. 

\nr{3} A Brieskorn homology $3$-sphere $\Sigma(a,b,c)$ admits a free $\cy p$-action if $p \nmid abc$. Does there exist a smooth, homologically trivial extension of this action with isolated fixed points to any  smooth simply connected  negative definite $4$-manifold $X$ with boundary $\Sigma(a,b,c)$ ?

\medskip
Anvari \cite{Anvari:2016a} proved that the free $\cy 7$ on $\Sigma (2,3,5)$ does not extend in this way over the minimal negative definite $4$-manifold obtained by resolving the link singularity. The method of proof involves studying equivariant Yang-Mills gauge theory on the non-compact manifold with cylindrical end obtained from $X$ by attaching the end $\Sigma(2,3,5) \times [0,\infty)$. Information about the Floer homology of Brieskorn spheres is an essential ingredient in tackling this problem (see Saveliev \cite{Saveliev:1999a,Saveliev:2000,Saveliev:2002b}).

 The corresponding extension problem from  $\bZ/p$ actions on Brieskorn spheres to smooth contractible $4$-manifolds was studied in \cite{Anvari:2016}, \cite{Anvari:2021}  where the orientation in the instanton moduli space played a key role (see also \cite{Kwasik:1993}).

\nr{4} What sets of rotation numbers can be realized by a smooth, pseudo-free $\cy n$-action  on $X = S^2 \times S^2$ ?

\medskip
A pseudo-free action is one with isolated singular points.  If the action is semi-free, there are ``standard models" with rotation data $\{(a, b), (c, d),  (a, -b), (c, -d)\}$ at the four fixed points. This question for $X = \CP^2$ was answered in \cite{Edmonds:1989a,Hambleton:1992}, where it was shown that the rotation data is always the same as for the linear actions arising from $PGL_3(\bC)$.

\nr{5} Let $(X, \pi)$ denote a smooth action of a finite group $\pi$ on a closed, simply connected smooth $4$-manifold.  Under what conditions does there exist a $\pi$-equivariant  principal $G$-bundle over $X$ with prescribed Chern classes, for $G = U(1)$ or $G=SU(2)$ ?

\medskip
Some information about this question was provided in \cite{Hambleton:1989,Hambleton:1992} for $X = \CP^2$ and $\pi$ finite cyclic, or more generally for $X$ negative definite (see also \cite{Hambleton:2010} for the connection between Chern classes and the isotropy representations). For $X = S^4$, the existence and classification of such bundles was applied by Austin \cite{Austin:1990} and Furuta \cite{Furuta:1989,Furuta:1990,Furuta:1990a}
 to study group actions via instanton gauge theory. The compactification of an equivariant version of Donaldson's Yang-Mills moduli space \cite{Donaldson:1990}, \cite{Hambleton:1992} involves ``bubbling" convergence to equivariant instantons over the $4$-sphere. Further applications of equivariant bundles arise in studying the equivariant compactification of moduli spaces over cylindrical end $4$-manifolds.

\section{The $G$-Signature Formula}\label{sec:two}

In this section we review the derivation of the $G$-signature formula as the index of the signature operator, see \cite[\S 18]{shanahan:1978}, \cite[pp.~585-586]{Atiyah:1968}.  The $G$-signature of a closed, smooth $4$-manifold $X$ with an orientation preserving action of a finite group $G$ acting as isometries on $X$ is defined as the virtual representation
]eqncount
\begin{equation}
\Sign(X,G) = [H_{+}^2(X;\bC)] - [H_{-}^2(X;\bC)]
\end{equation}
where $H_{\pm}^2(X;\bC)$ are the maximal positive/negative definite $G$-invariant subspaces of $H^2(X;\bC)$. Taking characters gives the $g$-signatures
\eqncount
\begin{equation}
\Sign(g,X) = \tr_g H_{+}^2(X) - \tr_g H_{-}^2(X)
\end{equation}
which by the $G$-signature formula can be computed in terms of certain Lefschetz numbers from the fixed point set.

Let $D:C^{\infty}(\Lambda^+) \rightarrow C^{\infty}(\Lambda^-)$ denote the signature operator. The Lefschetz numbers are computed as follows.  Let $X^g$ denote the fixed set of $g$, $n=\dim X^g$ and $N^g$ denotes the normal bundle of $X^g$ in $X$. Then the Lefschetz number is given by
\eqncount
\begin{equation}
L(g,D)=(-1)^{n(n+1)/2}\dfrac{\ch_g(\Lambda^+-\Lambda^-)(TX\vv_{X^g}\otimes \bC)Td(TX^g\otimes \bC)}{e(TX^g)\ch_g(\Lambda_{-1}N^g \otimes \bC)}[X^g]
\end{equation}

 where \begin{align*}
 &Td(TX^g \otimes \bC) = \dfrac{-x^2}{(1-e^{-x})(1-e^x)} \\
 &\ch(\Lambda^+ - \Lambda^-)(T^g \otimes \bC) = e^{-x} - e^x \\
 &\ch_g(\Lambda^+ - \Lambda^-)(N^g \otimes \bC) = e^{-y-i\theta} - e^{y+i\theta} \\
 & \ch_g(\Lambda_{-1}(N^g \otimes \bC)) = (1-e^{y+i\theta})(1 - e^{-y-i\theta}).
 \end{align*} 

 Here $x=e(TX^g)$ denotes the Euler class of the tangent bundle to the $2$-dimentional fixed point set, $\ch_g$ is the equivariant Chern character, and $\theta$ is the rotation on the normal bundle $N^g$ with Euler class $y$.  Note that $TX\vv_{X^g}=TX^g\oplus N^g$ and
\eqncount
\begin{equation}
\ch_g(\Lambda^+-\Lambda^-)(TX\vv_{X^g}\otimes \bC)=\ch_g(\Lambda^+-\Lambda^-)(TX^g \otimes \bC)\ch_g(\Lambda^+-\Lambda^-)(N^g \otimes \bC).
\end{equation}

 Substituting these expressions into the Lefschetz number and letting $F$ denote a $2$-dimensional fixed set component, then the contribution to the $g$-signature is given by
\begin{align*}
L(g,D)\vv_{F}&=(-1)\dfrac{(e^{-x}-e^x)(e^{-y-i\theta}-e^{y+i\theta})}{x(1-e^{y+i\theta})(1-e^{-y-i\theta})}\dfrac{x(-x)}{(1-e^{-x})(1-e^x)} [F]\\
&=\coth\left(\frac{y+i\theta}{2}\right)x\coth\left(\frac{x}{2}\right)[F]
\end{align*}
where the following trigonometric identity is used
\eqncount
\begin{equation}
\coth\left(\frac{x}{2}\right)=\frac{e^{-x}-e^x}{(1-e^{-x})(1-e^x)}.
\end{equation}
To evaluate on $[F]$, we use the Taylor expansions 

\begin{align*}
& x\coth\left(x/2\right)=2+\dfrac{x^2}{6\hphantom{i}}+\cdots \\
& \coth\left(\frac{y+i \theta}{2}\right)=\coth\left(i \theta/2\right)-\frac{1}{2}\csch^2\left(\frac{i \theta}{2}\right)y
\end{align*}
Thus the contribution to the Lefschetz number is given by 
\begin{align*}
L(g,D)\vv_{F}&=\{ 2 \coth(i\theta/2)- \csch^2(i \theta/2)y  \}[F]\\
&=-\csch^2(i\theta/2)[F]^2=\csc^2(\theta/2)[F]^2\\
&=\frac{-4t^{c_F}}{(t^{c_F}-1)^2}[F]^2
\end{align*}
where $\theta=\frac{2\pi c_F}{p}$ and $c_F$ is the rotation number on the normal fiber of $F$ and $t=e^{2\pi i/p}$ is a primitive $p$th root of unity. Similarly we can compute the contribution from isolated fixed points:
\begin{align*}
L(g,D)\vv_{pt}&=\frac{(e^{-i\theta_1}-e^{i\theta_1})(e^{-i\theta_2}-e^{i\theta_2})}{(1-e^{i\theta_1})(1-e^{-i\theta_1})(1-e^{i\theta_2})(1-e^{-i\theta_2})}\\
&=\coth(i\theta_1/2)\coth(i\theta_2/2)\\
&=-\cot(\theta_1/2)\cot(\theta_2/2)\\
&=\frac{(t^a+1)(t^b+1)}{(t^a-1)(t^b-1)}
\end{align*}
where $\theta_1=\frac{2\pi a}{p}$ and $\theta_2=\frac{2\pi b}{p}$ are the rotation numbers at the fixed point. By summing over the fixed point set components, the $G$-signature formula is given by 
\eqncount
\begin{equation}
\Sign(g, X)= \sum_i \dfrac{(t^{a_i}+1)}{(t^{a_i}-1)} \dfrac{(t^{b_i}+1)}{(t^{b_i}-1)}+\sum_j \dfrac{-4\alpha_j t^{c_j}}{(t^{c_j}-1)^2}
\end{equation}
where  $\alpha_j$ denotes the self-intersection $[F_j] \cdot[F_j]$ of the fixed $2$-spheres $\{F_j\}$.  Note the equivariant signature on the left-hand side of this formula is equal to the usual signature in this case because the action is homologically trivial.  We also remark here that this formula can be viewed as an equation in the cyclotomic field $\bQ[\zeta] = \bQ[t]/\Phi_p(t)$ where $\Phi_p(t)$ is the cyclotomic polynomial $ 1+t + t^2 + \cdots + t^{p-1}$.

\begin{remark} The $G$-signature formula (and its generalizations to the twisted $G$-signature formlas used in Section \ref{sec:four} and  Section \ref{sec:six} also hold in the topological locally linear category. The argument uses bordism theory  (see Wall \cite[Chapter 14B]{Wall:1999},  \cite{Gordon:1986}) and the foundational work of  Freedman \cite{Freedman:1982,Freedman:1984} 
(see also \cite[Theorem 1.1]{Kwasik:2017}). For the twisted $G$-signature formula \cite[Theorem 6.7]{Atiyah:1968a}, the analytic index and the topological index coincide as homomorphisms 
$K_G(TX) \to R(G)$. If $G$ is a finite group of odd order, Madsen \cite[Theorem 3.3 and Proposition 3.6]{Madsen:1986} relates equivariant $K_G$-theory to equivariant bordism after localizing away from 2.
\end{remark}

\section{Congruence Relations for the $G$-Signature Formula}\label{sec:three}

In this section we derive congruence relations satisfied by the rotation data $\mathcal{F}$ using the $G$-signature formula. We first note that since $(t^a - 1)/(t-1)$ is a unit in ring of cyclotomic integers $\bZ[\zeta]$ when $(a,p)=1$, multiplying both sides of the $G$-signature formula by $(t-1)^2$ induces an equation in the ring $R = \bZ[\zeta]$. The $I$-adic expansion of the resulting right-hand side leads to congruence relations relating the rotation data, where $I$ denote the ideal generated by $(t-1)$ in $R$.

 Following the method of \cite{Hambleton:1989} we lift the equation to $\bZ[t]$, compute the Taylor expansions about $t=1$ and reduce the coefficients modulo $p$. Since the indeterminacy of the coefficients are determined from the expansion of the cyclotomic polynomial $\Phi_p(t)$ (for which $p$ divides the coefficients of its Taylor expansion about $t=1$ up to order $p-1$) we obtain valid congruence relations by equating coefficients modulo $p$ up to order $p-1$.

The expansion arising from contributions from isolated fixed points are given by
\begin{align*}
\dfrac{(t^{a}+1)}{(t^{a}-1)} \dfrac{(t^{b}+1)}{(t^{b}-1)}(t-1)^2&=\frac{4}{ab}+\frac{4}{ab}(t-1)+\frac{1}{3}\left(\frac{a^2+b^2+1}{ab}\right)(t-1)^2\\
&-\frac{1}{180}\left(\frac{a^4+b^4-5a^2b^2+3}{ab}\right)(t-1)^4+\cdots
\end{align*}
Similarly the expansion of the second term is given by expressions of the form
\begin{align*}
\dfrac{-4\alpha t^c}{(t^{c}-1)^2}(t-1)^2 &=\frac{-4\alpha}{c^2}+\frac{-4\alpha}{c^2}(t-1)+\frac{1}{3}\frac{\alpha(c^2-1)}{c^2}(t-1)^2\\
&-\frac{1}{60}\frac{\alpha(c-1)(1+c+c^2+c^3)}{c^2}(t-1)^4+\cdots.
\end{align*}
Equating both sides of the expansion $\pmod p$ from the resulting equation in $R$ we thus obtain the following congruence relations:

\begin{theorem}\cite[p.~625]{Hambleton:1989}
Let $(X,\pi)$ denote a simply connected, closed $4$-manifold with a homologically trivial, locally-linear group action of a finite cyclic group $\pi=\bZ/p$ of odd order. Then the following congruence relations hold
\smallskip
\begin{enumerate}
\item $\sum_i \frac{1}{a_ib_i} - \sum_j \frac{\alpha_j}{c_j^2} \equiv 0 \pmod p$\\
\item $\sum_i \frac{a_i^2+b_i^2}{a_ib_i} + \sum_j \alpha_j \equiv 3\Sign(X) \pmod p$ \\
\item $\sum_i \frac{a_i^4+b_i^4-5a_i^2b_i^2}{a_ib_i} + 3\sum_j \alpha_j c_j^2   \equiv 0 \pmod p$ \\
\item $\sum_i \frac{2a_i^6-7a_i^4b_i^2-7a_i^2b_i^4+2b_i^6}{a_ib_i} +10\sum_j\alpha_jc_j^4 \equiv 0 \pmod p$.
 \end{enumerate}
 Higher-order relations are valid up to and including terms of order $p-1$.
\end{theorem}

\begin{example}[Linear models on $\bC P^2$]\label{ex:fourtwo}
Let $G=\bZ/p$ with odd prime $p$ act linearly on $X=\bC P^2$ by $t \cdot [z_1:z_2:z_3]=[\zeta^az_1:z_2:z_3]$ for $0< a<p$. The fixed point set consists of one isolated fixed point $[1:0:0]$ with tangential rotation number $(a,a)$ and a fixed projective line $F=\{ [z_1:z_2:z_3] \vv z_1 = 0\}$ with self-intersection $+1$ and a rotation of $c_F \equiv a \pmod p$ on the normal bundle. Then it is easy to check  that the congruence relations are satisfied. Similarly, in the case when the action is given by $t \cdot [z_1:z_2:z_3]=[\zeta^az_1:\zeta^bz_2:z_3]$ for $0<a<b<p$ the action consists of three isolated fixed points $[0:0:1], [1:0:0], [0:1:0]$ with rotation numbers $(a,b), (b-a, -a), (a-b, -b)$ and with some algebra it can be checked that the congruence relations hold. Additional examples can be obtained  for $\# _n\bC P^2 $ by equivariant connected sums along fixed point sets using the linear models.
\end{example}

\section{Equivariant Line Bundles}\label{sec:four}

Let $(X,\pi)$ denote a simply connected, closed $4$-manifold with a homologically trivial action of a finite cyclic group $\pi=\bZ/p$ of odd prime $p$ and $L\rightarrow X$ an equivariant line bundle.  The index of the signature operator twisted by an equivariant line bundle is the same as in the previous section, except that one has to consider the contribution of the equivariant Chern character of the line bundle restricted to the fixed point set. Let $F$ denote a fixed surface as before, then the contribution to the Lefschetz number is given by: 
\begin{align*}
L(g,D)\vv_{F}&=\{ 2 \coth(i\theta/2)- \csch^2(i \theta/2)y \}\{ \ch_g(i^{\ast}L) \}[F]\\
&=\{ 2 \coth(i\theta/2)- \csch^2(i \theta/2)y \}\{ e^{z+i\phi} \} [F]\\
&=\{ 2 \coth(i\theta/2)- \csch^2(i \theta/2)y \}\{ e^{i\phi}+e^{i\phi}z  \}[F]\\
&=\{ 2 \coth(i\theta/2)e^{i\phi}z - \csch^2(i \theta/2)y e^{i\phi}\\
&=2c_1(i^{\ast}L)[F]\dfrac{(t^{c_F}+1)}{(t^{c_F}-1)}t^{\lambda}+\dfrac{-4[F]^2 t^{c_F}}{(t^{c_F}-1)^2}t^{\lambda}
\end{align*}
where $\phi=\frac{2\pi \lambda}{p}$  is the rotation on the fiber in $L$ over a fixed point on $F$ and $z=c_1(i^{\ast}L)$  denotes the restriction of the first Chern class to the fixed set component. Similarly for the contribution to the isolated fixed points. To summarize, given an equivariant line bundle $(L,\pi)\rightarrow (X,\pi)$, the index of the twisted $G$-signature operator gives a virtual character given by

\begin{equation*}
\chi(t)=\sum_i \dfrac{(t^{a_i}+1)}{(t^{a_i}-1)} \dfrac{(t^{b_i}+1)}{(t^{b_i}-1)}t^{\lambda_i}+\sum_j \dfrac{-4\alpha_j t^{c_j}}{(t^{c_j}-1)^2}t^{\lambda_j}+\sum_j 2c_1(i^{\ast}L)[F_j]\dfrac{(t^{c_j}+1)}{(t^{c_j}-1)}t^{\lambda_j}
\end{equation*}
\noindent
where $\{F_j\}$ are fixed $2$-spheres of the action on $X$ with $\alpha_j$ denoting the self-intersection $[F_j] \cdot[F_j]$. Note that 
 
\begin{align}
\chi(1)&=\ch(L)\cL(X)[X]=\left(1+c_1(L)+\frac{1}{2}c_1(L)^2\right)\left(4+\frac{p_1}{3}\right)[X]\\
&=\left(\frac{p_1}{3}+2c_1(L)^2\right)[X]=\Sign(X)+2c_1(L)^2[X].
\end{align}
\noindent
Since $\chi(t)$ is a virtual character for $G=\bZ/p$ we may write $\chi(t)=\sum_{i=0}^{p-1} a_i t^i$ for some $a_i \in \bZ$ and 
\begin{align*}
\chi(t)(t-1)^2=\chi(1) (t-1)^2 + \text{higher order terms}\pmod{p}
\end{align*}
\noindent
We can then take the Taylor expansion of the right-hand side of the twisted $G$-signature formula after multiplying by $(t-1)^2$ and equate the first and second order terms to obtain two additional congruence relations. The first order term vanishes while the second order term is congruent to $\chi(1) \pmod p$. Taking Taylor expansions for these terms in the $G$-signature formula give:
\begin{align*}
\dfrac{(t^{a}+1)}{(t^{a}-1)} \dfrac{(t^{b}+1)}{(t^{b}-1)}(t-1)^2t^{\lambda}&=\frac{4}{ab} + \frac{4(\lambda+1)}{ab}(t-1) +\\
&\frac{1}{3}\frac{(a^2+b^2+1+6\lambda^2+6\lambda)}{ab}(t-1)^2 +\cdots.
\end{align*}
\noindent
Similarly for the second term:
\begin{align*}
\dfrac{-4\alpha t^c}{(t^{c}-1)^2}(t-1)^2t^{\lambda} &=\frac{-4\alpha}{c^2}+\frac{-4\alpha(1+\lambda)}{c^2}(t-1) + \\
&\frac{1}{3}\frac{\alpha(c^2-1-6\lambda-6\lambda^2)}{c^2}(t-1)^2+ \cdots.
\end{align*}
\noindent
and for the third term:
\begin{align*}
2m\dfrac{(t^c+1)}{(t^{c}-1)}(t-1)^2t^{\lambda} &=\frac{4m}{c}(t-1)+\frac{2m(2\lambda+1)}{c}(t-1)^2+ \cdots
\end{align*}
where $m=c_1(i^{\ast}L)[F]$. 
Combining these expressions we obtain the following theorem:
\begin{theorem}\label{thm:fiveone}
Let $(L,\pi)\rightarrow (X,\pi)$ denote an equivariant line bundle over a simply connected, closed $4$-manifold with a homologically trivial group action of a finite cyclic group $\pi=\bZ/p$ of odd order. Then the following congruence relation holds
\begin{align*}
 & (i)\  \sum_i \frac{\lambda_i}{a_ib_i} - \sum_j \frac{\lambda_j \alpha_j}{c_j^2}+\sum_j \frac{c_1(i^{\ast}L)[F_j]}{c_j} \equiv 0 \pmod p \\
 & (ii)\  \sum_i \frac{\lambda_i^2}{a_ib_i} - \sum_j \frac{\lambda_j^2 \alpha_j}{c_j^2}+2\sum_j \frac{\lambda_j c_1(i^{\ast}L)[F_j]}{c_j} \equiv c_1(L)^2[X] \pmod p .
\end{align*}
\end{theorem}

\begin{example}[Linear models on $\bC P^2$]
Let $p$ denote an odd prime and consider $X=\bC P^2$ and a linear action $t \cdot [z_1:z_2:z_3]=[\zeta^az_1:\zeta^b z_2:z_3]$ for $0 <a<b<p$. We give an explicit construction of equivariant line bundles over $X$. Consider a finite dimensional complex representation  space $V=\bC(\lambda_1)\oplus \bC(\lambda_2) \oplus \bC(\lambda_3)$ with action given by $\rho  \in GL_3(\bC)$:
\eqncount
\begin{equation}
\rho = \begin{pmatrix} 
  t^{\lambda_1}    & &  \\ 
       & t^{\lambda_2} &  \\
       & & t^{\lambda_3}\\
\end{pmatrix}\
: \bC^3\setminus \{0\}  \longrightarrow \bC^3 \setminus \{0\}
\end{equation}
with the $\lambda_i$ positive integer weights. Let $S(V)$ denote the unit sphere in $V$ then $\rho$ commutes with the free $S^1$-action on $S(V)$ and
\begin{align*}
\bC P^2=S(V)/S^1=S^5/S^1.
\end{align*}
The $\rho$-action on $S(V)$ is a lift of the linear action on $X$ if the following system of linear congruences 
\begin{align*}
& \lambda_1-\lambda_3 \equiv a , \quad   \lambda_2 -\lambda_3 \equiv b \\
& \lambda_2-\lambda_1 \equiv b-a,  \quad \lambda_3-\lambda_1 \equiv -a \\
& \lambda_1-\lambda_2\equiv a-b ,\quad \lambda_3 -\lambda_2 \equiv -b.
\end{align*}
has a solution. This system has one degree of freedom; let $\lambda_3 = \lambda$ be a fixed parameter, then the isotropy representations over the three isolated fixed points are given by:
\begin{align*}
&\text{fixed point}~p_1=[0:0:1], \text{rotation number} ~(a,b),  \text{with isotropy} \lambda_3\equiv\lambda\\
&\text{fixed point}~p_2=[1:0:0], \text{rotation number} ~(b-a,-a), \text{with isotropy} \lambda_1\equiv\lambda + a\\
&\text{fixed point}~p_3=[0:1:0], \text{rotation number} ~ (-b,a-b) \text{with isotropy} \lambda_2\equiv\lambda + b.
\end{align*}
The equivariant line bundle 
\begin{align*}
L=S(V)\times_{S^1}\bC
\end{align*}
is the canonical bundle over $\bC P^2$ and the congruence relations of the theorem are satisfied:
\begin{align*}
&\sum_i \frac{\lambda_i}{a_i b_i}\equiv 0 \\
&\sum_i \frac{\lambda_i^2}{a_i b_i} \equiv 1.
\end{align*}
In the case when $b=0$ the action is given by $t\cdot [z_1:z_2:z_3]=[\zeta^a z_1: z_2:z_3]$ and has a fixed projective line $F=\{ z_1 =0\}$ with normal rotation number $c_F \equiv a \pmod p$ and self-intersection $+1$, while the isolated fixed point $[1:0:0]$ has rotation number $(a,a)$. The compatibility for a lift of the linear action is the congruence relation:
\begin{align*}
\lambda - \lambda_F \equiv a \pmod p.
\end{align*}
It is easily seen that the congruence relation of the theorem are satisfied:
\begin{align*}
\frac{\lambda}{a^2}-\frac{\lambda_F}{a^2}+\frac{c_1(i^{\ast}L)[F]}{a}\equiv 0.
\end{align*}
where we used $c_1(i^{\ast}L)[F]=-1$,  since the first Chern class of the canonical line bundle over $\bC P^2$ is negative  of the preferred generator $[F] \in H^2(\bC P^2;\bZ)$ \cite[Theorem 14.10, p.~169]{Milnor:1974}.
Similarly, the second relation is 
\begin{align*}
\frac{\lambda^2}{a^2} -\frac{\lambda_F^2}{a^2}+\frac{2\lambda_F c_1(i^{\ast}L)[F]}{a} 
&\equiv \frac{(a+\lambda_F)^2-\lambda_F^2}{a^2}+\frac{2\lambda_F c_1(i^{\ast}L)[F]}{a} \\
&\equiv 1 \equiv c_1( L)^2[X].
\end{align*}
\end{example}

\begin{example}
Let $\pi=\bZ/p$ act on $X=\bC P^2$ preserving an almost complex structure. Then the complexified second exterior power of the cotangent bundle $K_X = \Lambda^2_{\bC}T^{\ast}X$ is an equivariant line bundle see \cite[Proposition 1.8]{Hambleton:1989}). We check the congruence relations in the case when the action has a fixed projective line and leave the remaining case of isolated fixed points to the reader. 

The action from the linear model is given by  $t \cdot [z_1:z_2:z_3]=[\zeta^az_1:z_2:z_3]$ for some $0< a<p$ and the fixed point set consists of an isolated fixed point with tangential rotation number $(-a,-a)$ and a fixed projective line 
$$F=\{ [z_1:z_2:z_3] \vv z_1 = 0\}$$
 with self-intersection $+1$ and a rotation of $c_F \equiv a \pmod p$ on the normal bundle. 
The isotropy representations over the fixed point set are $K_X \vv_{pt} = t^{2a}$ and $K_X \vv_F = t^{-a}$. With these values substituted it is easy to verify the congruence relations in Theorem \ref{thm:fiveone} for $L = K_X$, noting that $K_X$ is the canonical class of $X$ with dual $-3$ times the standard generator in $H_2(\bC P^2; \bZ)$.
\end{example}

\begin{definition} Let $(X, \pi)$ be a homologically trivial of $\pi=\cy p$ in the setting of Theorem A, with rotation data 
$\mathcal{F}=\{(a_i,b_i),(c_j,\alpha_j)\vv i \in I, j \in J \}.$
  We say that $(X, \pi)$ \emph{satisfies the condition of Theorem A} if there exists a set of isotropy data $\mathcal{I} =\{\lambda_i,\lambda_j \vv i \in I, j \in J\}$,  and a set of integers  $\{m_j  \vv j \in J\}$ so that the equation given in Theorem A holds.
\end{definition}

In the next statement, we apply the equivariant connected sum operation to line bundles. 

\begin{lemma}\label{lem:fivetwo} Suppose that $(X,\pi)$ satisfies the condition of Theorem \textup{A}. There there is an equivariant connected sum $(X \sharp\, \CP^2, \pi)$ with a linear $\pi$-action on $\CP^2$ which also  satisfies the condition of Theorem \textup{A}.
\end{lemma}

\begin{proof} We first suppose that $(X, \pi)$ contains a fixed $2$-sphere $F_j$ with data $\{\alpha_j, c_j\}$. Let $(\CP^2, \pi)$ be the linear action given by  $t \cdot [z_1:z_2:z_3]=[\zeta^{-c_j}z_1:z_2:z_3]$ as in Example \ref{ex:fourtwo}. We do the equivariant connected sum (preserving the orientations) along a point in $F_j$ and a point in the fixed $2$-sphere of $\CP^2$. The new data is obtained by (i) adding the data $\{(-c_j, -c_j); \lambda_j\}$ for the newly created isolated fixed point (on $\CP^2$), and (ii) the data $\{(c_j, \alpha_j + 1); \lambda_j\}$ for the  new fixed $2$-sphere. With these choices, it follows that the action $(X \sharp\, \CP^2, \pi)$ satisfies the condition of Theorem A. The proof in case the action $(X, \pi)$ has only isolated fixed points is easier, and will be left to the reader.
\end{proof}

\begin{remark} Suppose that $(X, \pi)$  has  data  satisfying the condition of Theorem A, and contains a fixed $2$-sphere $F$. Let $X_0\subset X$ denote the complement of a  linear $\pi$-invariant $4$-ball neighbourhood of a point $x \in F$. If  $L$ is an equivariant line bundle over $(X \sharp\, \CP^2, \pi)$,  then the restriction of $L$ to $X_0$ extends to an equivariant line bundle over $(X, \pi)$ realizing the given data.
\end{remark}
\section{The proof of Theorem A}\label{sec:five}
The first relation in Theorem \ref{thm:fiveone} proves the necessary conditions of Theorem A. To prove sufficiency we will need the following lemmas. 

Note that in a standard lens space $Y=L(n;a,b)$, a generator $\mu \in H_1(Y; \bZ)$ is represented by a circle fibre in the fibration $S^1 \to L(n;a,b) \to S^2$ given by the quotient of a free $\cy n$ action on $S^3$.

\begin{lemma}\label{lem:fivethree}
The linking paring $lk\colon H_1(Y)\times H_1(Y) \rightarrow \bQ/\bZ$ in the lens space $Y=L(n;a,b)$ is given by $lk(\mu,\mu)=\dfrac{ab}{n}$ where $\mu$ is a generator of $H_1(Y;\bZ)=\bZ/n$.
\end{lemma}
\begin{proof}
In the usual representation of lens spaces $L(n;q)$ as the quotient of the free $\bZ/n$ action $t\cdot(z_1, z_2) = (\zeta z_1, \zeta^q z_2)$ on $S^3$,  the linking pairing is given by $lk(\mu,\mu)=\dfrac{q}{n}$. The diffeomorphism $L(n;a^{\ast}b) \rightarrow L(n;a,b)$  where $a^\ast$ is the multiplicative inverse mod $n$ arising from changing the generator in $\bZ/n$ induces a map on first homology given by multiplication by $a$. It follows that the linking pairing  on $Y$ is given by 
$$lk(a\mu,a\mu)=a^2\cdot \left(\dfrac{a^{\ast}b}{n}\right)=\dfrac{ab}{n} \in \bQ/\bZ.
$$
\vskip-15pt
\end{proof}
\begin{lemma}[{\cite[Proposition 1.4, p.~621]{Hambleton:1989}, \cite[Lemma 2.11, p.~95]{Hedden:2011}}]
Let $\phi\colon\pi_1(Y) \longrightarrow U(1)$ be the holonomy representation of a flat $U(1)$-bundle over the lens space $Y=L(n;a,b)$ that sends a generator $\mu$ to $\exp(2\pi i \lambda/n)$. Then the Poincar\'{e} dual of the first Chern class $PD(c_1(L))$ is given by $\dfrac{\lambda}{ab}[\mu]$ in $H_1(Y;\bZ)$. 
\end{lemma}
\begin{proof}
The adjoint to the linking form $\Phi\colon H_1(Y;\bZ) \rightarrow \Hom(H_1(Y;\bZ), \bQ/\bZ)$ sends  $m[\mu]$ to $lk(m\mu, -)$ which can be identified with  the holonomy representation of the flat bundle via :
$$e^{2\pi i \cdot lk(m\mu, -)}\colon H_1(Y;\bZ) \rightarrow U(1)$$
and this maps the generator to $\exp(2\pi i m\cdot (ab/n))$. It follows that $m\equiv \dfrac{\lambda}{ab} \pmod n$.
\end{proof}
If $Y$ is a lens space, we let
$\hat\mu \in H^2(Y;\bZ)$ denote a standard cohomology generator: the Poincar\'e dual to the circle fibre class $\mu \in H_1(Y;\bZ)$.

\begin{lemma}\label{lem:fivefour}
Let  $u\colon Y'\to Y$  be a $d$-fold regular covering of lens space, where $H_1(Y;\bZ) \cong \cy{dn}$ and $H_1(Y';\bZ) \cong \cy n$, with $\gcd(d, n) = 1$. Then $u^*(\hat\mu) = \hat\mu' \in H^2(Y'; \bZ)$.
\end{lemma}
\begin{proof}
For a $d$-fold regular covering $u\colon Y'\to Y$ of lens spaces, we have $u_*[\mu'] = d[\mu] \in H_1(Y, \bZ)$ and $u_*[Y'] = d[Y]\in H_3(Y;\bZ)$. The cohomology generator
$\hat\mu = [Y]\cap \mu \in H^2(Y;\bZ)$ is the Poincar\'e dual of $\mu$, and similarly for $\hat\mu' \in H^2(Y';\bZ)$. We have the formula
$$u_*([Y']\cap u^*(\hat\mu)) = u_*(\mu') =  d\mu.$$
 If $H_1(Y;\bZ) = \cy{dn}$ and $H_1(Y';\bZ) = \cy n$, where $\gcd(d, n) = 1$, then $u^*(\hat\mu) = k\hat\mu'$ implies that $k \equiv 1 \pmod n$. Hence $u^*(\hat\mu) = \hat\mu' \in H^2(Y';\bZ)$.
\end{proof}

\parni
Suppose that $(X,\pi)$ satisfies the assumptions of Theorem \textup{\ref{thm:fiveone}}. Let $\Sigma\subset X$ denote the singular set of the action, and let $X_0 := X - \Sigma$. Write $\alpha_j  = p^{a_j}\ba_j$, where $\overline{\alpha_j}$ is prime to $p$ (for each fixed $2$-sphere $F_j$).

\begin{lemma}\label{lem:fivesix}  If the singular set $\Sigma \subset X$ contains an isolated point, then $H_1(X_0;\bZ)$ is a quotient of $\bigoplus \cy {\ba_j}$ and  has order prime to $p$.
\end{lemma}
\begin{proof}  Let $F \subset X$ denote the disjoint union of the  fixed $2$-spheres, so $F = \bigcup_j F_j$. First note that $H_1(X_0) \cong H^3(X, F)$ and we have an exact sequence
$$ \dots \to H^2(X) \to H^2(F) \to H^3(X, F) \to H^3(X) \to \dots$$
Since 
$H^3(X) = 0$, and the homology classes of the fixed $2$-spheres are linearly independent mod $p$ (by \cite[Corollary 2.6]{Edmonds:1989}),  it follows 
 that $H^3(X, F)\cong H_1(X_0)$ is a torsion group of order prime to $p$. 
 
Moreover, the exact Mayer-Vietoris sequence
$$ 0 \to H_2(X_0) \oplus H_2(F) \to H_2(X) \to H_1(\bd X_0) \to H_1(X_0) \to 0$$
and the equality $H_1(\bd X_0) = H_1(\bd\nu(F))  \cong \bigoplus \cy {\alpha_j}$ completes the proof.
\end{proof}

\begin{proof}[The proof of Theorem A]
By Theorem \ref{thm:fiveone}(i), the indicated formulas hold if $(X,\pi)$ admits an equivariant line bundle. 
It remains to prove that a solution $\{ \lambda_i, \lambda_j, m_j \vv i \in I, j\in J \}$ to the congruence relation 
\eqncount
\begin{equation}\label{eqn:congrence}
\sum_i \frac{\lambda_i}{a_ib_i} - \sum_j \frac{\lambda_j \alpha_j}{c_j^2}+\sum_j \frac{m_j}{c_j} \equiv 0 \pmod p 
\end{equation}
is sufficient for the existence of an equivariant line bundle with  $\{\lambda_i, \lambda_j\}$ isotropy representations over the isolated fixed points and $2$-spheres respectively, and $m_j = c_1(i^{\ast} L \vv_{F_j} )$. \\

For simplicity, we will assume that the action $(X,\pi)$ contains at least one isolated fixed point. This may always be arranged by taking the equivariant connected sum of $(X,  \pi)$ along a fixed $2$-sphere with a suitable linear $\pi$-action on $\CP^2$  (see Lemma \ref{lem:fivetwo}).

Let $X_0 = X-N$, where $N = \nu(\Sigma)$ is a $\pi$-invariant tubular neighbourhood of the singular set $\Sigma \subset X$. More explicitly,  $X_0$ is the compact $4$-manifold with boundary obtained  by removing $\pi$-invariant $4$-balls around each isolated fixed point and $\pi$-invariant tubular neighbourhoods $D^2 \rightarrow \nu(F_j) \rightarrow F_j$ around each $\pi$-fixed $2$-sphere $F_j$ with rotation $t^{c_j}$ on $D^2$-fibers. Then $\nu(F_j)$ is a $2$-disk bundle over $S^2$ with Euler class $\alpha_j [F] \in  H^2(F;\bZ)\cong \bZ$, and the lens space
 $\bd \nu(F_j) = L(\alpha_j, 1)$ inherits a free $\cy p$ action with rotation number $c_j$ on the circle fibre.

If $W_0 :=X_0/\pi$ denotes the quotient manifold with (regular) covering map $q\colon X_0 \to W_0$ classified by $u \colon W_0 \to B\pi$, then the boundary $\partial W_0$ consists of lens spaces $Y_i=L(p;a_i, b_i)$ and $Y_j=L(p\alpha_j; c_j,c_j)$. Note that $$H_1(W_0; \bZ) \cong \cy p \oplus H_1(X_0, \bZ)$$ by the spectral sequence of the covering. By Lemma \ref{lem:fivesix}, $H_1(X_0;\bZ)$ is a quotient of 
$\bigoplus \cy {\ba_j}$  and has order prime to $p$.

Recall that $\pi$-equivariant line bundles $L$ over  $(X, \pi)$ are classified by  an element
$$\theta (L)\in H^2_\pi(X; \bZ) = H^2(X\times_\pi E\pi; \bZ)$$
in the Borel  equivariant cohomology of $X$
 (see \cite{Lashof:1983}). Since the $\pi$-action on $X_0$ is free, for the restriction $L_0 \searrow X_0$ we have
$\theta(L_0) \in H^2_\pi(X_0; \bZ) \cong H^2(W_0; \bZ)$, and $\theta(L_0) = c_1(\bar L_0)$, where $\bar L_0$ is the line bundle over $W_0$ obtained by dividing out the free $\pi$-action on the total space of $L_0$.
Moreover, in the short exact sequence
$$ 0 \to H^2(\cy p; \bZ) \xrightarrow{c^*} H^2(W_0;\bZ) \xrightarrow{q^*} H^2(X_0; \bZ) \to 0$$
the pullback $q^*(\theta(L_0)) = c_1(q^*(\bar L_0)) = c_1(L_0) \in H^2(X_0; \bZ)$.

The strategy will be to find a suitable element $\theta (L)\in H^2_\pi(X; \bZ) $ by studying the Mayer-Vietoris sequence
$$ \dots \to H^2_\pi( X)  \to H^2_\pi(X_0) \oplus H^2_\pi(N) \to H^2_\pi(\bd X_0) \xrightarrow{\delta} H^3_\pi(X) \to  H^3_\pi(X_0) \oplus H^3_\pi(N)\to \dots$$
in Borel cohomology associated to the $\pi$-equivariant decomposition $X = X_0 \cup N$.

\medskip
{\bf Obervations}:
\begin{enumerate}\label{obs}
\item\label{obs:one} The Mayer-Vietoris coboundary map $H^2_\pi(\bd X_0) \xrightarrow{\delta} H^3_\pi(X)$ factors as 
$$H^2_\pi(\bd X_0) \xrightarrow{\delta} H^3_\pi(X_0, \bd X_0) \cong H^3_\pi(X, N) \to H^3_\pi(X).$$
\item \label{obs:two}The cokernel of the map $ H^2_\pi(N) \to H^2_\pi(\bd X_0)$ has exponent $p$. This follows from the commutative diagram of restriction maps
$$\xymatrix@R-10pt@C-10pt{ H^2_\pi(N) \ar[r] \ar[d] &  H^2_\pi(\bd X_0) \ar[d] \ar[r]  &  \bigoplus \cy{p\alpha_j} \ar[d]\\
 H^2(N) \ar[r] & H^2(\bd X_0)\ar[r]^\cong &    \bigoplus \cy{\alpha_j}
}$$
since the map $H^2_\pi(N)  \to H^2(N)$ is surjective (by the Borel spectral sequence) and the map $H^2(N, \bd N) \to H^2(N)$ is adjoint to the (diagonal) intersection form on $N$, with cokernel $H^2(\bd N) = H^2(\bd X_0)$, hence determined by the self-intersection numbers $\{\alpha_j\}$.
\item \label{obs:three} We have  $H^3_\pi(X_0, \bd X_0) \cong H^3(W_0, \bd W_0)\cong H_1(W_0) \cong \cy p \oplus H_1(X_0)$, where 
$ H_1(X_0)$ is a quotient of 
$\bigoplus \cy {\ba_j}$  and has order prime to $p$ (by Lemma \ref{lem:fivesix}).
\item \label{obs:four} The Mayer-Vietoris coboundary map $H^2_\pi(\bd X_0) \xrightarrow{\delta} H^3_\pi(X)$ also  factors as 
$$H^2_\pi(\bd X_0) \xrightarrow{\delta} H^3_\pi(N, \bd X_0) \cong H^3_\pi(X, X_0) \to H^3_\pi(X).$$
From the first three points  above, it follows that $\Image (\delta) = \cy p \subseteq H^3_\pi(X)$ may be identified with the first summand of the identification $H^3_\pi(X_0, \bd X_0) \cong  \cy p \oplus H_1(X_0)$.
\end{enumerate}

\medskip
To complete the proof of Theorem A, it is now enough to produce a class 
$$\theta_0 \in H^2_\pi(X_0)  \cong H^2(W_0),$$
which added
together with the classes already found in $H^2_\pi(N)$ will have image zero in $H^2_\pi(\bd X_0)$. By the observations above, this amounts to finding a $U(1)$-bundle $\bar L_0$ on $\bd W_0$ whose first Chern class $\theta_0 = c_1(\bar L_0)
$ has image of order prime to $p$ under the coboundary map
$$H^2(\bd W_0) \to  H^3(W_0, \bd W_0).$$
In other words, we need to find suitable $U(1)$-bundles over each of boundary components of $\bd W_0$, so that the sum of their first Chern classes is zero (mod $ p$). This required relation is exactly the condition \eqref{eq:thma} given in the statement of Theorem A.

\medskip
Let $Y$ denote one of the lens spaces in $\partial W_0$, For convenience, we will identify $\cy n \cong H^2(Y;\bZ)$ by $a \mapsto a\cdot\hat\mu$, for $a \in \cy n$, where $\hat\mu \in H^2(Y;\bZ)$ denotes a standard generator, Poincar\'e dual to the circle fibre class $\mu \in H_1(Y;\bZ)$ (introduced in Lemma \ref{lem:fivethree}).

If $Y_i=L(p;a_i,b_i) \subset \bd W_0$ arises from one of the isolated fixed points in $X$,  then choose the holonomy representation that sends a generator in $\pi_1(Y)$ to $\exp(2\pi i \lambda_i/p) $. The first Chern class of the associated  flat $U(1)$-bundle is 
$$\dfrac{\lambda_i}{a_ib_i}[\hat\mu_i] \in H^2(Y_i;\bZ) \cong \cy p.$$
In the case when the $\pi$-action on $X$ has only isolated fixed points, the condition for an extension is that these elements lie in the kernel of $\delta $ in $H^2(\partial W_0;\bZ)=\bigoplus_i H^2(Y_i;\bZ)$ which is equivalent to the condition  $\sum_i \dfrac{\lambda_i}{a_ib_i} \equiv 0 \pmod p$. 

In the general case  when components of the fixed set contain $2$-spheres, we need to consider contributions from the lens spaces $Y_j=L(p\alpha_j;c_j,c_j)$. These lens spaces arise from the free $t^{c_j}$-action on $\widetilde Y_j := \partial \nu(F_j) \iso L(\alpha_j;1)$. Consider the induced covering spaces of $Y_j$ by the lens spaces $L(p^{a_j+1},1)$ and 
$L(\ba_j,1)$, where $\alpha_j = p^{a_j}\ba_j$, and note that the covering maps induce an isomorphism:
\eqncount
\begin{equation}\label{eq:fiveseven}
\vcenter{\xymatrix@R-10pt{ H^2(Y_j; \bZ)  \ar[r]^(0.3)\cong\ar[d]^\cong &  H^2(L(p^{a_j+1},1)) \oplus H^2(L(\ba_j,1)) \ar[d]^\cong \\
   \cy{p{\alpha_j}}  \ar[r]^\cong &   \cy{p^{a_j+1}} \oplus \cy {\ba_j}
}}
\end{equation}
Under the two covering maps, the standard cohomology generator $\hat\mu_j \in H^2(Y_j;\bZ) $  is sent to the standard generators in $H^2(L(p^{a_j+1},1))$ and 
$ H^2(L(\ba_j,1))$,  respectively, by Lemma \ref{lem:fivefour}. The maps in the lower sequence are the reductions mod $p^{a_j+1}$ and $\ba_j$, after using the identifications provided by the cohomology generators.

The required Chern class  $c_1(\bar L_j)=\dfrac{\ell_j}{c_j^{2}}[\hat\mu_j]\in H^2(Y_j;\bZ)$ for each component $Y_j$ can now be determined uniquely by solving the congruences:
\eqncount
\begin{equation}\label{eq:fivenine}
\vcenter{\xymatrix@R-28pt@C-20pt{
\ell_j& \equiv& -\lambda_j\alpha_j \pmod{p^{a_j + 1}} ,\\
\ell_j& \equiv& c_jm_j \pmod{\ba_j}.}} 
\end{equation}
and hence we have a $U(1)$-bundle $\bar L_j \searrow Y_j$.
The minus sign is chosen in the first congruence because the induced orientation on $Y_j$ from $\bd W_0$ is opposite to its orientation as the disk bundle over $S^2$ with Euler class $\alpha_j$. 

By diagram \eqref{eq:fiveseven} and Lemma \ref{lem:fivefour},   the  first Chern class has image
$$c_1(\bar L_j) = \dfrac{\ell_j}{c_j^{2}}[\hat\mu_j] =
  \dfrac{-\lambda_j \alpha_j +
c_jm_j}{c_j^2}[\hat\mu_j] \in H^2(Y_j;\bZ)$$
with respect to the decomposition $H^2(Y_j;\bZ) \cong \cy{p^{a_j+1} }\oplus\bZ/\ba_j $. After substituting these expressions into the congruence relation \eqref{eqn:congrence}, we see that 
the sum vanishes mod $p$ (from Observation \ref{obs:four} above). Hence we have
\begin{equation*}
\delta ( \sum_i \dfrac{\lambda_i}{a_ib_i} [\hat\mu_i] + \sum_j \dfrac{\ell_j}{c_j^2} [\hat\mu_j]  ) = 0
\end{equation*}
under the coboundary map $\delta \colon H^2_\pi(\bd X_0) \to H^3_\pi(X)$, and
the required line bundle $\bar L_0$ over $W_0$ exists.
\end{proof}

\section{Equivariant $SU(2)$ Bundles}\label{sec:six}

In this section we compute a (necessary) congruence relation similar to the previous section,  but for equivariant $SU(2)$-bundles. As above, we work over a closed, simply connected, oriented $4$-manifold with a finite homologically trivial cyclic group action.  We again use the twisted $G$-signature formula (with the previously established notation).  In particular,  let $D$ denote the signature operator twisted by an  equivariant $SU(2)$-bundle $E \longrightarrow X$, then the contribution to the Lefschetz numbers from isolated fixed points is given by
\begin{align*}
L(g,D)\vv_{pt}=\frac{(t^{a}+1)}{(t^{a}-1)} \dfrac{(t^{b}+1)}{(t^{b}-1)}(t^{\lambda}+t^{-\lambda}).
\end{align*}
We need to compute the contribution from isolated fixed $2$-spheres $F$. Since $E\vv_F=L\oplus L^{-1}$, we have $\ch_g(L\oplus L^{-1}\vv_F)=\{ e^{\lambda+z}+e^{-\lambda-z} \}[F]$ and
\begin{align*}
L(g,D)\vv_F&=\{ 2\cot(i\theta/2)-\csch^2(i\theta/2)y \}\ch_g(L\oplus L^{-1}\vv_F)[F]\\
&=\{ 2\frac{(t^c+1)}{(t^c-1)}-\frac{4t^c y}{(t^c-1)^2} \} \{ e^{\lambda}(1+z)+e^{-\lambda}(1-z) \}[F]\\
&=\{ 2\frac{(t^c+1)}{(t^c-1)}-\frac{4t^c y}{(t^c-1)^2} \} \{ t^{\lambda}+t^{-\lambda}+z(t^{\lambda}-t^{-\lambda}) \}[F] \\
&=-\frac{4t^c[F]^2}{(t^c-1)^2}(t^{\lambda}+t^{-\lambda}) + 2c_1(L)[F]\frac{(t^c+1)}{(t^c-1)} (t^{\lambda}-t^{-\lambda}).             
\end{align*}
Also note 
\begin{align*}
\chi(1)&=\ch(E)\cL(X)[X]=(2-c_2(E))(4+\frac{1}{3}p_1)\\
&=2\Sign(X)-4c_2(E).
\end{align*}
We now again multiply both sides of the $G$-signature formula by $(t-1)$, take Taylor expansions about $t=1$ and reduce coefficients modulo $p$:
\begin{align*}
\dfrac{(t^{a}+1)}{(t^{a}-1)} \dfrac{(t^{b}+1)}{(t^{b}-1)}(t-1)^2(t^{\lambda}+t^{-\lambda})&=\frac{8}{ab}+\frac{8}{ab}(t-1)+\\
&\frac{2}{3}\frac{(a^2+b^2+1+6\lambda^2)}{ab}(t-1)^2 +\cdots.
\end{align*}
and for the second term, where we let $m$ denote $c_1(L)[F]$:
\begin{align*}
&(t-1)^2\{ \dfrac{-4\alpha t^c}{(t^{c}-1)^2}(t^{\lambda}+t^{-\lambda})+2m\dfrac{(t^c+1)}{(t^{c}-1)}(t^{\lambda}+t^{-\lambda}) \} \\
&=\frac{-8\alpha}{c^2}+\frac{-8\alpha}{c^2}(t-1)+\frac{2}{3}\frac{(\alpha c^2-\alpha-6\alpha \lambda^2+12m c \lambda)}{c^2}(t-1)^2+\cdots
\end{align*}
Summing over all the fixed sets and simplifying the coefficient of second order term $(t-1)^2$, we obtain: 
\begin{align*}
2\Sign(X)+\sum_i \frac{4 \lambda_i^2}{a_ib_i} -\sum_j \frac{4 \alpha_j \lambda_j^2}{c_j^2}+\sum_j \frac{ 8m_j \lambda_j}{c_j}.
\end{align*}
Equating this with $\chi(1)=2\Sign(X)-4c_2(E)$ and reducing coefficients modulo $p$ gives the following congruence relation:

\begin{theorem}\label{thm:sevenone}
Let $(E,\pi)\rightarrow (X,\pi)$ denote an equivariant $SU(2)$-bundle over a simply connected, closed $4$-manifold with a homologically trivial group action of a finite cyclic group $\pi=\bZ/p$ of odd prime order. Then the following congruence relation holds
\begin{align*}
& \sum_i \frac{\lambda_i^2}{a_ib_i} - \sum_j \frac{\alpha_j\lambda_j^2}{c_j^2}+\sum_j \frac{2\lambda_j}{c_j}c_1(i^{\ast}L_j)[F_j] \equiv -c_2(E)[X] \pmod p, 
\end{align*}
where $L_j$ is a local reduction $E\vv_{F_j}=L_j \oplus L_j^{-1}$.
\end{theorem}

Instanton gauge theory on $c_2(E)=1$ bundles in the equivariant setting provide many examples of equivariant $SU(2)$-bundles on smooth definite $4$-manifolds. We next check these relations for some examples of smooth cyclic group actions in the well-known linear models (see 
\cite{Austin:1990,Braam:1993,Furuta:1989,Furuta:1990,Furuta:1990a,Hambleton:1995,Hambleton:2004}).

\begin{example}[Linear Models on $S^4$]
Let $X=S^4$ with a linear $\cy p$-action which gives rotation numbers $(a, b)$ and $(a, -b)$. Let $E$ denote the instanton one equivariant $SU(2)$-bundle, i.e. with $c_2(E)=1$. Then the congruence relation is given by
\begin{align*}
-c_2(E)=\frac{\lambda_1^2}{ab}-\frac{\lambda_2^2}{ab} \pmod p
\end{align*}
It is elementary to check that this congruence relation is satisfied with the following isotropy representations
\begin{align*}
&\lambda_1=\frac{b-a}{2}\\
&\lambda_2=\frac{a+b}{2}.
\end{align*}
over the fibres of the fixed points.
\end{example}

\begin{example}[Linear Models on $\overline{\bC P}^2$]
Let $X=\overline{\bC P}^2$ with a linear $\cy p$-action with one isolated fixed point with rotation number $(a, -a)$ for some $a \pmod p$ and a fixed projective line $F$ with rotation number $a \pmod p$ on the normal bundle. Let $E$ again denote the instanton one equivariant $SU(2)$-bundle. The congruence relation gives
\begin{align*}
-1\equiv -\frac{\lambda^2}{a^2} + \frac{\lambda_F^2}{a^2}+\frac{2m\lambda_F}{a}
\end{align*}
where $m=c_1(i^{\ast}L)[F] $.
There exists two distinct lifts giving rise to equivariant bundles which admit $G$-invariant ASD connections. In the case when the equivariant lift comes from "bubbling" on the isolated fixed point then $m=0$ and 
\begin{align*}
&\lambda \equiv a \pmod p \\
&\lambda_F \equiv 0 \pmod p.
\end{align*}
Thus the congruence is satisfied. On the other hand, if we choose the equivariant lift associated to the fixed $2$-sphere (from 3-dimensional fixed connected component in the moduli space of equivariant ASD connections with $c_2(E)=1$) then $m=-1$ and 
\begin{align*}
&\lambda \equiv a/2 \pmod p\\
&\lambda_F \equiv a/2 \pmod p,
\end{align*}
again the congruence relation is satisfied.
\end{example}

\begin{remark} At present we do not have general sufficient conditions for the existence of equivariant principal $SU(2)$-bundles. However, in the special cases where $X$ is negative definite and $c_2(E) =1$, for any choice of reduction $E = L \oplus L^{-1}$ from a cohomology class $c_1(L) = \alpha \in H^2(X;\bZ)$ with $\la \alpha^2, [X] \ra = -1$, we have sufficient conditions by applying Theorem A.
\end{remark}

\section{Equivariant Index Computation}\label{sec:seven}
In this section we compute the dimension of the moduli space of invariant anti-self dual connections for a given equivariant $SU(2)$-bundle over smooth $4$-manifolds with a given homologically trivial cyclic group action. 

Let $X$ be a simply connected, closed, smooth negative definite $4$-manifold,   with a homologically trivial action of a finite group $G$.   If $E \searrow X$ is an $SU(2)$-bundle with $c_2(E) = k$,  the moduli space $\cM^\ast_1(X)$ of  irreducible ASD connections (on an $SU(2)$-bundle $E$ with $c_2(E) = 1$) inherits a $G$-action,  and the connected components of the fixed point set $\cM^G_1(X)$ correspond to $G$-invariant ASD connections for certain equivariant lifts of the $G$-action on $X$ to $E$ (see \cite{Furuta:1989}, 
 \cite{Braam:1993}, \cite[Theorem A]{Hambleton:1992},  \cite[\S 2]{Hambleton:2004}).

 We want to compute the dimension of the moduli space $\cM^G_k(X)$ of irreducible $G$-invariant ASD connections. This is motivated by Example \ref{ex:eightone}, for which the formal dimension $\dim \cM^\ast_1(X) = 5$. In this case, we expect a dimension formula that gives $1$ and $3$-dimensional strata depending on contributions from isolated fixed points or isolated fixed $2$-spheres in $X$ and on the isotropy representations from the equivariant lift (see \cite{Braam:1993} and \cite{Hambleton:1995} for details). There are similar index calculations in the literature in various gauge-theoretic settings (for example, see \cite[\S 3]{Fintushel:1985}, \cite{Austin:1990}), \cite{Lawson:1988,Lawson:1993}, \cite{Wang:1993},  \cite{Anvari:2016a}).

We first very briefly review the dimension calculation in the non-equivariant setting to set some notation. Let $D_A^+=d_A^{\ast}+d_A^{+}\colon \Omega^1(\ad E) \rightarrow \Omega^{0}(\ad E) \oplus \Omega^2_{+}(\ad E)$ denote the anti-self duality operator, and let $\cM_k$ denote the ASD moduli space with $c_2(E)=k$. Note that the formal dimension is given by $\dim \cM_k=- \Ind(D_A^+)$ and this is given by 
\eqncount
\begin{equation}
\Ind(D_A^+)=\hat{A}(X)\ch(S^+)\ch(\ad_{\bC} E) [X]
\end{equation}
where $S=S^+\oplus S^-$ and $\hat{A}(X)=\prod \frac{x_i/2}{\sinh(x_i/2)}$ with $\ch(S)=\prod (e^{x_i/2}+e^{-x_i/2})$ and $\ch(S^+)-\ch(S^-)=\prod (e^{x_i/2}-e^{-x_i/2})$. Using this we compute 
\begin{align*}
2\hat{A}(X)\ch(S^+)\ch(\ad_{\bC}E)[X]&=(4+\frac{1}{3}p_1+\chi)(3-4c_2(E))[X]\\
&=-16c_2(E)+3(\frac{p_1}{3}+\chi).
\end{align*}
Thus the index $\Ind(D_A^+)=-8c_2(E)+\frac{3}{2}(\Sign+\chi)(X)$ and we get the usual expression $\dim \cM_k=8k-3/2(\chi+\Sign)(X)$ for the dimension of the moduli space.
Also note the following alternative expression for the index:
\begin{align*}
\Ind(D_A^+)&=\frac{\ch(S^+-S^-)\ch(S^+)\ch(\ad_{\bC}E)Td(TX\otimes \bC)}{e(X)}[X]\\
&=\hat{A}(X)\ch(S^+\otimes \ad_{\bC}E)[X].
\end{align*} 
For the equivariant setting $E$ is an equivariant $SU(2)$-bundle and let $D=D_A^+$ denote the anti-self duality operator $d_A^{\ast}+d_A^{+}: \Omega^1(\ad E)^{G} \rightarrow \Omega^{0}(\ad E)^{G} \oplus \Omega^2_{+}(\ad E)^{G}$.
We compute the equivariant index by averaging the Lefschetz numbers as in \cite{Fintushel:1985}:
\begin{equation*}
\Ind(D)=\frac{1}{p}\sum_{g\in G}L(g,D)
\end{equation*}
\begin{align*}
\Ind(D)&=\frac{1}{p}\{ L(1,D) + \sum_{g \neq 1} L(g,D)\}\\
&=\frac{1}{p}\{-8c_2(E)+\frac{3}{2}(\chi+\Sign)(X) + \sum_{g \neq 1} L(g,D)\}\\
&=\frac{1}{p}\{-8c_2(E)+\frac{3p}{2}(\chi+\Sign)(X/G)-\frac{3}{2}(d_{\chi}+d_{\sigma})(X^{G})+\sum_{g \neq 1} L(g,D)\}
\end{align*}
where $p\chi(X/G)=\chi(X)+d_{\chi}$ with $d_{\chi}=\sum_{g\neq 1}\chi(X^{g})$ is the Euler characteristic defect terms and similarly for the signature defect term: 
\begin{align*}
&-\frac{3}{2}(d_{\chi}+d_{\sigma})[pt]=-\frac{3}{2}(1-\cot(\theta_1/2)\cot(\theta_2/2))\\
&-\frac{3}{2}(d_{\chi}+d_{\sigma})[F]=-\frac{3}{2}(2+[F]^2\csc^2(\theta/2)),
\end{align*}
where $(\theta_1, \theta_2)$ are the rotation numbers at an isolated fixed point and $\theta=c_F$ is the rotation number on the normal bundle to $F$. Decomposing the contributions from isolated fixed points and $2$-spheres:
\begin{align*}
\sum_{g\neq 1}L(g,D)(X^G)=\sum_{g\neq 1}\{ \sum_{i} L(g,D)\vv_{(a_i,b_i)}+ \sum_{j} L(g,D)\vv_{F_j}\}.
\end{align*}
Now $\ch_g(\ad_{\bC}E)(pt)=3-4\sin^2(\frac{\pi k \ell}{p})$, with $\ell$ the isotropy representation on the fiber of $E$ over the fixed point. The Lefshetz numbers from the fixed sets can be computed directly from the index formula and are given by:
\eqncount
\begin{equation}
L(g,D)\vv_{pt}=\frac{-1}{2}[\cot(\theta_1/2)\cot(\theta_2/2)-1]\ch_g(\ad_{\bC}E) [pt]
\end{equation}
\eqncount
\begin{equation}
L(g,D)\vv_{F}=[-i\cot(\theta/2)+\frac{1}{2}(\chi+\csc^2(\theta/2)y)]\ch_g(\ad_{\bC}E) [F],
\end{equation}
with $\chi$ the Euler class of the tangent bundle to $F$ and $y$ is the Euler class of the normal bundle to $F$. We first compute the contribution from isolated fixed points.
\begin{align*}
L(g,D)\vv_{pt}&=\frac{-1}{2}[\cot(\theta_1/2)\cot(\theta_2/2)-1][3-4\sin^2(\frac{\pi k \ell}{p})] [pt]\\
&=-\frac{3}{2}[\cot(\theta_1/2)\cot(\theta_2/2)-1]-2\sin^2(\frac{\pi k \ell}{p})\\
&+2\cot(\theta_1/2)\cot(\theta_2/2)\sin^2(\frac{\pi k \ell}{p}).
\end{align*}
Summing over all isolated fixed points gives 
\begin{align*}
\frac{1}{p}\sum_{g\neq 1} \sum_{i} L(g,D)\vv_{(a_i,b_i)}&=\frac{3}{2p}\sum_i (d_{\chi}+d_{\sigma})(a_i,b_i)-\frac{2}{p}\sum_i \sum_{k=1}^{p-1} \sin(\frac{\pi k \ell_i}{p})\\
&+\frac{2}{p}\sum_i \sum_{k=1}^{p-1} \cot(\frac{a_i \pi k}{p})\cot(\frac{b_i \pi k}{p})\sin^2(\frac{\pi k \ell_i}{p})\\
&=\frac{3}{2p}\sum_i (d_{\chi}+d_{\sigma})(a_i,b_i) + m + \sum_i \rho L(p,a_i,b_i,\ell_i)
\end{align*}
where $m$ is the number isolated fixed points with non-trivial representation on the fiber and $\rho L(p,a,b,\ell)$ is the rho invariant of lens spaces.

We need to compute  $\ch_g(\ad_{\bC}E\vv_{F})$. Since an $SU(2)$ bundle restricted over a fixed $2$-submanifold has a local abelian reduction $E\vv_{F}=L\oplus L^{-1}$ for some $L$, we have $\ad E\vv_{F}= L^2\oplus \underline{\bR}$. We need to compute $\ch_g(\ad_{\bC}E\vv_{F})=\ch_g(L^2)+\ch_g(\overline{L^2})+1$ and this contributes
\begin{align*}
\ch_g(\ad_{\bC}E\vv_{F})&=(g+gc_1(L^2))+(g^{-1}+g^{-1}c_1(\overline{L^2}))+1\\
&=(g+g^{-1}+1)+c_1(L^2)(g-g^{-1})\\
&=(3-4\sin^2(\frac{\pi k \ell}{p}))+2ic_1(L^2)\sin(\frac{2\pi k \ell}{p}), 
\end{align*}
where now $\ell$ is the isotropy representation on the fibre over the fixed $2$-sphere $F$. Substituting these terms, the Lesfchetz number $L(g,D)\vv_{F}$ evaluated on fixed $2$-spheres gives:
\begin{align*}
L(g,D)\vv_{F} &=[-i\cot(\theta/2)+\frac{1}{2}(\chi+\csc^2(\theta/2)y)][(3-4\sin^2(\frac{\pi k \ell}{p}))+2ic_1(L^2)\sin(\frac{2\pi k \ell}{p})][F]\\
&=\frac{1}{2}[\chi+\csc^2(\theta/2)y][3-4\sin^2(\frac{\pi k \ell}{p})]+2c_1(L^2)\sin(\frac{2\pi k \ell}{p})\cot(\frac{\theta}{2}).
\end{align*}
Let us introduce a kind of rho invariant term for fixed surfaces: 
\begin{align*}
\rho_F(\ell)=\frac{2}{p}\sum_{k=1}^{p-1} \csc^2(\frac{\pi c_F k}{p})\sin^2(\frac{\pi k \ell}{p}) [F]^2-\frac{4c_1(L)[F]}{p}\sum_{k=1}^{p-1} \sin(\frac{2\pi k \ell}{p})\cot(\frac{\pi k c_F}{p}),
\end{align*}
with this notation we have 
\begin{align*}
\frac{1}{p}\sum_{g \neq 1}\sum_{j} L(g,D)\vv_{F_j}&=\frac{3}{2p}\sum_j(d_{\chi}+d_{\sigma})[F_j]-\frac{2}{p}\sum_j \chi(F_j)\sum_{k=1}^{p-1} \sin^2(\frac{\pi k \ell_j}{p})\\
&-\sum_j \rho_{F_j}(\ell_j).
\end{align*}
Now combining all the terms we obtain:
\begin{align*}
\Ind(D_A)&=\frac{-8}{p}c_2(E)+\frac{3}{2}(\chi+\Sign)(X/G))-m+\sum_{i}\rho L(p,a_i,b_i,\ell_i)\\
& -\sum_{j\: \text{with} \: \ell_j \neq 0} \chi(F_j)-\sum_j \rho_{F_j}(\ell_j).
\end{align*}
Since $\dim \cM^G_k(X)=-\Ind(D_A)$, the dimension formula is
\begin{align*}
\dim \cM^G_k(X)&=\frac{8}{p}c_2(E)-\frac{3}{2}(\chi+\Sign)(X/G)+m -\sum_{i}\rho L(p,a_i,b_i,\ell_i)\\
& +\sum_{j\: \text{with} \: \ell_j \neq 0} \chi(F_j)+\sum_j \rho_{F_j}(\ell_j).
\end{align*}
Before giving an example we note a few special cases. When the action on $X$ only has isolated fixed points, let $(a_i,b_i)$ denote the rotation numbers and $\ell_i$ the isotropy representation over the points, the formula reduces to the following:
\begin{align*}
\dim \cM^G_k(X)=\frac{8c_2(E)}{p}-\frac{3}{2}(\chi+\Sign)(X/G)+m -\sum_{i}\rho L(p,a_i,b_i,\ell_i).
\end{align*}
For invariant ASD connections on the four-sphere this formula reduces to that of \cite[p. 394]{Austin:1990}. In the case of $SO(3$)-bundles in the orbifold setting,  see Fintushel and Stern \cite{Fintushel:1985}. When the action on $X$ is a smooth involution with fixed $2$-sphere and non-trivial action on fibre $c_F \equiv \ell \equiv 1$ mod 2 the formula above reduces to:
\begin{align*}
\dim \cM^G_k(X)&=4c_2(E)-\frac{3}{2}(\chi+\Sign)(X/G)+\chi(F)+ [F]^2.
\end{align*}
which matches with Wang \cite[Theorem 18, p.~130]{Wang:1993}.
We finish this section with an example.

\begin{example}\label{ex:eightone}
Let $X= \#_3 \overline{\bC P}^2$ with a linear $\cy p$-action with $ p=5$ that arises from equivariant connected sums of linear actions in the following way. Take the equivariant connected sum of two copies of $\overline{\bC P}^2$ along the two dimensional fixed sets which fixes a projective line and a rotation number of $(1, -1)$ at the isolated fixed points in each copy. Now at one of the isolated fixed points take the equivariant connected sum with $\overline{\bC P}^2$ that has a linear action with $3$ isolated fixed points with rotation numbers $(1, 1), (2, -1), (2, -1)$. 

The result is a smooth, homologically trivial $\cy 5$-action on $X$ that has $3$ isolated fixed points with rotation data $\{(1, -1), (2, -1), (2, -1)\}$ and a single fixed $2$-sphere $F$ with rotation number $c_F \equiv 1 \pmod p$ on the normal bundle and has self intersection $-2$. 

The compactified, equivariant  ASD instanton one moduli space $\cM_1(X)$ has dimension 5 with fixed components that are $1$ and $3$-dimensional which correspond to invariant ASD connections for a lifted action to the $SU(2)$-bundle (see \cite{Hambleton:1995}). 

The boundary of the moduli space is the "bubbling" of highly concentrated ASD connections which can be identified with a copy of $X$. The isolated fixed points propagate  $1$-fixed dimensional strata into the moduli space. We will compute the dimension of these strata using the dimension formula from this section and from the fixed point data.  

For example, at the isolated fixed point $(2, -1)$ the highly concentrated instantons correspond to ASD connections on the $4$-sphere, with equivariant lifts matching the linear models which then pull back to $X$ using the degree $1$-map in the formation of the Taubes boundary. This determines the equivariant lift on $X$ and has isotropy representation $t^{\lambda_1}$ over the fixed point $(2, -1)$ with $\lambda_1 \equiv -3 \pmod p$ and $t^{\lambda_2}$ over all the other fixed point sets with $\lambda_2 \equiv 1 \pmod p$. The dimension formula gives:
\begin{align*}
\frac{8}{p}-\rho L(p, 2, -1, -3) - \rho L(p, 2, -1, 1) - \rho L(p, 1, -1, 1) +\chi(F)+ \rho_F(1) = 1.
\end{align*}
On the other hand, at a point on the fixed $2$-sphere $F$ following the same procedure with the degree one Taubes map, we can pull-back an equivariant bundle from the linear model on $S^4$ with a fixed embedded $2$-sphere. This time we get an equivariant $SU(2)$-bundle on $X$ with $c_1(L)[F]=-1$ in the local reduction $ E\vv_{F} = L \oplus L^{-1}$ (in this case actually a global reduction as it corresponds to reducible). The isotropy representation is $t^{\lambda}$ over all the fixed point sets with $\lambda \equiv 1 \pmod p$. We then have:
\begin{align*}
\frac{8}{p}-2\rho L(p, 2, -1, 1) - \rho L(p, 1, -1, 1) +\chi(F)+ \rho_F(1) = 3.
\end{align*}
after substituting the data  into the dimension formula.
\end{example}


\providecommand{\bysame}{\leavevmode\hbox to3em{\hrulefill}\thinspace}
\providecommand{\MR}{\relax\ifhmode\unskip\space\fi MR }
\providecommand{\MRhref}[2]{%
  \href{http://www.ams.org/mathscinet-getitem?mr=#1}{#2}
}
\providecommand{\href}[2]{#2}

\end{document}